\documentclass{article}
\usepackage{amsmath,amsthm,amssymb}
\usepackage{hhline}

\usepackage[colorlinks=true,linkcolor=blue,citecolor=blue,pdfpagelabels=false]{hyperref}

\usepackage{thmtools}
\theoremstyle{plain}
  \declaretheorem[numberwithin=section]{theorem}

  \declaretheorem[numberlike=theorem]{lemma}
  \declaretheorem[numberlike=theorem]{conjecture}
\theoremstyle{definition}
  \declaretheorem[numberlike=theorem]{example}
  
\newenvironment{acknowledgements}{\bigskip\textbf{Acknowledgements.}}{}

\newcommand{\email}[1]{{\textit{Email:} \texttt{#1}}}
\newcommand{\nin}{\not\in}

\begin{document}

\title{Core partitions into distinct parts and an analog of Euler's
theorem}
\author{Armin Straub\thanks{\email{straub@southalabama.edu}}\\
Department of Mathematics and Statistics\\
University of South Alabama}

\date{January 26, 2016}

\maketitle

\begin{abstract}
A special case of an elegant result due to Anderson proves that the number of
$( s, s + 1)$-core partitions is finite and is given by the Catalan number
$C_s$. Amdeberhan recently conjectured that the number of $( s, s + 1)$-core
partitions into distinct parts equals the Fibonacci number $F_{s + 1}$. We
prove this conjecture by enumerating, more generally, $( s, d s - 1)$-core
partitions into distinct parts. We do this by relating them to certain tuples
of nested twin-free sets.

As a by-product of our results, we obtain a bijection between partitions into
distinct parts and partitions into odd parts, which preserves the perimeter
(that is, the largest part plus the number of parts minus $1$). This simple
but curious analog of Euler's theorem appears to be missing from the
literature on partitions.
\end{abstract}

\section{Introduction}

A {\emph{partition}} $\lambda$ of $n$ (for very good introductions see
\cite{andrews-part} and \cite{ae-partitions}) is a finite sequence $(
\lambda_1, \lambda_2, \ldots, \lambda_{\ell})$ of positive integers $\lambda_1
\geq \lambda_2 \geq \cdots \geq \lambda_{\ell}$ such that
$\lambda_1 + \lambda_2 + \ldots + \lambda_{\ell} = n$. The integers
$\lambda_1, \lambda_2, \ldots, \lambda_{\ell}$ are referred to as the
{\emph{parts}} of $\lambda$, with $\lambda_1$ being the largest part and
$\ell$ the number of parts. Such a partition $\lambda$ is frequently
represented by its {\emph{Young diagram}}, which we take to be a
left-justified array of square cells with $\ell$ rows such that the $i$th row
consists of $\lambda_i$ cells. To each cell $u$ is assigned a {\emph{hook}},
which is composed of the cell $u$ itself as well as all cells to the right of
$u$ and below $u$. The {\emph{hook length}} of $u$ is the number of cells the
hook consists of. A partition $\lambda$ is said to be {\emph{$t$-core}} if
$\lambda$ has no cell of hook length equal to $t$. An explanation of this
terminology is given, for instance, in \cite{ahj-cores}. More generally,
$\lambda$ is said to be $( t_1, t_2, \ldots, t_r)$-core if $\lambda$ is
$t$-core for $t = t_1, t_2, \ldots, t_r$.

The motivation to count partitions that are $t$-core for different values of
$t$ has been sparked by the following elegant result due to Anderson
\cite{anderson-cores}.

\begin{theorem}
  \label{thm:anderson}The number of $( s, t)$-core partitions is finite if and
  only if $s$ and $t$ are coprime. In that case, this number is
  \begin{equation*}
    \frac{1}{s + t}  \binom{s + t}{s} .
  \end{equation*}
\end{theorem}

In particular, the number of $( s, s + 1)$-core partitions is the Catalan
number
\begin{equation*}
  C_s = \frac{1}{s + 1} \binom{2 s}{s} = \frac{1}{2 s + 1}  \binom{2 s +
   1}{s},
\end{equation*}
which also counts the number of Dyck paths of order $s$. Generalizations to $(
s, s + 1, \ldots, s + p)$-core partitions, including a relation to generalized
Dyck paths, are given in \cite{al-cores}.

In a different direction, Ford, Mai and Sze \cite{fms-cores} show that the
number of self-conjugate $( s, t)$-core partitions is
\begin{equation*}
  \binom{\lfloor s / 2 \rfloor + \lfloor t / 2 \rfloor}{\lfloor s / 2
   \rfloor},
\end{equation*}
provided that $s$ and $t$ are coprime. More generally, Amdeberhan
\cite{amdeberhan-conj} raises the interesting problem of counting the number
of special partitions which are $t$-core for certain values of $t$. In
particular, he conjectures the following count.

\begin{conjecture}
  \label{conj:amdeberhan}The number of $( s, s + 1)$-core partitions into
  distinct parts equals the Fibonacci number $F_{s + 1}$.
\end{conjecture}

It is further conjectured in \cite{amdeberhan-conj} that the largest
possible size of an $( s, s + 1)$-core partition into distinct parts is
$\lfloor s ( s + 1) / 6 \rfloor$, and that there is a unique such largest
partition unless $s \equiv 1$ modulo $3$, in which case there are two
partitions of maximum size. Amdeberhan also provides a conjecture for the
average size of these partitions. We do not pursue these more intricate, but
very interesting, questions here (the interested reader is referred to, for
instance, \cite{os-core}, \cite{ahj-cores}, \cite{sz-core},
\cite{chw-core}, \cite{johnson-core}, and the references therein, for the
case of general core partitions, and \cite{xiong-cores} for the case of $(
s, s + 1)$-core partitions into distinct parts). Instead, we focus on the most
basic question on core partitions into distinct parts, namely to enumerate
them. Ultimately, our main result is the following enumeration of $( s,
t)$-core partitions into distinct parts for a two-parameter family of values
$( s, t)$.

\begin{theorem}
  \label{thm:fibx:distinct}Let $d, s \geq 1$. The number $N_d ( s)$ of $(
  s, d s - 1)$-core partitions into distinct parts is characterized by $N_d (
  1) = 1$, $N_d ( 2) = d$ and, for $s \geq 3$,
  \begin{equation*}
    N_d ( s) = N_d ( s - 1) + d N_d ( s - 2) .
  \end{equation*}
\end{theorem}

In particular, the case $d = 1$ clearly settles
Conjecture~\ref{conj:amdeberhan}. This special case has since been also
independently proved by Xiong \cite{xiong-cores}. Before giving a proof of
Theorem~\ref{thm:fibx:distinct} in Section~\ref{sec:fibx}, we discuss an
elementary bijective proof of the special case $d = 1$ in
Sections~\ref{sec:fib} and \ref{sec:euler}.

We do so, because a natural extension of our approach leads to a simple but
curious analog of Euler's theorem on partitions into distinct (respectively
odd) parts, which appears to be missing from the literature on partitions.
Namely, we obtain a bijection between partitions into distinct parts on the
one hand and partitions into odd parts on the other hand, which preserves the
perimeter of the partitions. Here, following Corteel and Lovejoy
\cite[Section~4.2]{cl-overpartitions} (up to a shift by $1$), we refer to
the {\emph{perimeter}} of a partition as the maximum part plus the number of
parts minus $1$ (equivalently, the perimeter of $\lambda$ is the maximum hook
length in $\lambda$).

\begin{theorem}
  \label{thm:distinct-odd}The number of partitions into distinct parts with
  perimeter $M$ is equal to the number of partitions into odd parts with
  perimeter $M$. Both are enumerated by the Fibonacci number $F_M$.
\end{theorem}

\begin{example}
  The partitions into distinct parts with perimeter $5$ are $( 5)$, $( 4, 1)$,
  $( 4, 2)$, $( 4, 3)$ and $( 3, 2, 1)$. The partitions into odd parts with
  perimeter $5$ are $( 5)$, $( 3, 3, 3)$, $( 3, 3, 1)$, $( 3, 1, 1)$ and $( 1,
  1, 1, 1, 1)$. In each case, there are $F_5 = 5$ many of these partitions.
\end{example}

While it appears natural, we have been unable to find the result in
Theorem~\ref{thm:distinct-odd} in the existing literature. On the other hand,
an intriguingly similar result of Euler is widely known: the number $D ( n)$
of partitions of $n$ into distinct parts equals the number $O ( n)$ of
partitions of $n$ into odd parts. In other words, there is a bijection between
partitions into distinct and odd parts, which preserves the size of the
partitions. While there are bijective proofs (see, for instance,
\cite[Chapter~2.3]{ae-partitions}), Euler famously proved his claim using a
very elegant manipulation of generating functions (see, for instance,
\cite[Cor.~1.2]{andrews-part} or \cite[Chapter~5.2]{ae-partitions}).
Namely, he observed that
\begin{eqnarray*}
  \sum_{n \geq 0} D ( n) x^n & = & ( 1 + x) ( 1 + x^2) ( 1 + x^3)
  \cdots\\
  & = & \frac{1 - x^2}{1 - x}  \frac{1 - x^4}{1 - x^2}  \frac{1 - x^6}{1 -
  x^3} \cdots\\
  & = & \frac{1}{1 - x}  \frac{1}{1 - x^3}  \frac{1}{1 - x^5} \cdots =
  \sum_{n \geq 0} O ( n) x^n .
\end{eqnarray*}
Several refinements of Euler's theorem due to Sylvester, Fine and
Bousquet-M\'elou--Eriksson are beautifully presented, for instance, in the
book \cite[Chapter~9]{ae-partitions} by Andrews and Eriksson.

\begin{example}
  Bousquet-M\'elou and Eriksson \cite{bme-hall}, \cite{bme-hall2} show
  that the number of lecture hall partitions of $n$ with length $k$ (these are
  special partitions of $n$ into distinct parts) is equal to the number of
  partitions of $n$ into odd parts with each part at most $2 k - 1$. Among
  other refinements, they also prove that the number of partitions of $n$ into
  distinct parts with sign-alternating sum $k$ is equal to the number of
  partitions of $n$ into $k$ odd parts. A corresponding combinatorial
  bijection is given in \cite{ky-do}.
\end{example}

\begin{example}
  Another refinement, found in \cite[(23.91)]{fine}, shows that the number
  of partitions of $n$ into distinct parts with maximum part $M$ is equal to
  the number of partitions of $n$ into odd parts such that the maximum part
  plus twice the number of parts is $2 M + 1$.
\end{example}

\begin{example}
  The {\emph{rank}} of a partition is the difference between the largest part
  and the number of parts. Note that the rank of a partition into distinct
  parts is always nonnegative. Then, we have \cite[(24.6)]{fine} that the
  number of partitions of $n$ into odd parts with maximum part equal to $2 M +
  1$ is equal to the number of partitions of $n$ into distinct parts with rank
  $2 M$ or $2 M + 1$.
\end{example}

A natural question is whether similarly interesting refinements exist for
Theorem~\ref{thm:distinct-odd}, that is, for partitions into distinct
(respectively odd) parts with perimeter $M$.

\section{Partitions with bounded hook lengths}\label{sec:fib}

We begin by proving the case $d = 1$ of Theorem~\ref{thm:fibx:distinct}, thus
establishing Conjecture~\ref{conj:amdeberhan}. The proof for the general case
is then given in Section~\ref{sec:fibx}. Let $F_s$ denote the Fibonacci
numbers with $F_0 = 0$ and $F_1 = 1$.

\begin{theorem}
  \label{thm:fib:distinct}There are $F_{s + 1}$ many $( s, s + 1)$-core
  partitions into distinct parts.
\end{theorem}

In preparation for Theorem~\ref{thm:fib:distinct}, we first prove the
following claim.

\begin{lemma}
  \label{lem:distinct:m1}A partition into distinct parts is $( s, s + 1)$-core
  if and only if it has perimeter strictly less than $s$.
\end{lemma}

\begin{proof}
  Recall that the perimeter of a partition $\lambda$ is the maximum hook
  length in $\lambda$. We therefore need to show that, if $\lambda$ is $( s, s
  + 1)$-core, then $\lambda$ is $( s, s + 1, s + 2, \ldots)$-core. Suppose
  otherwise, and let $t$ be the smallest hook length in $\lambda$ larger than
  $s$. By construction, $\lambda$ is $( t - 1, t - 2)$-core. Consider the
  Young diagram of $\lambda$, and focus on a cell $u$ with hook length $t$. A
  moment of reflection reveals that, since $\lambda$ has distinct parts, the
  cell to the right of $u$ has hook length $t - 1$ or $t - 2$. This
  contradicts the fact that $\lambda$ is $( t - 1, t - 2)$-core, and so our
  claim must be true.
\end{proof}

\begin{proof}[Proof of Theorem~\ref{thm:fib:distinct}]
  By virtue of
  Lemma~\ref{lem:distinct:m1}, we need to show that there are $F_{s + 1}$ many
  partitions into distinct parts with perimeter strictly less than $s$. One
  checks directly that this is true for $s = 1$ (in which case, $F_2 = 1$ and
  the relevant set of partitions consists of the empty partition only) and $s
  = 2$ (in which case, $F_3 = 2$ and the relevant set of partitions consists
  of the partition $( 1)$ and the empty partition).
  
  Let $s \geq 3$ and, for the purpose of induction, suppose that the
  number of partitions into distinct parts with perimeter less than $r$ is
  given by $F_{r + 1}$ for all $r < s$. Let $\lambda = ( \lambda_1, \lambda_2,
  \ldots)$ be a partition into distinct parts with perimeter less than $s$.
  Then exactly one of the following two cases applies:
  \begin{enumerate}
    \item The largest part $\lambda_1$ satisfies $\lambda_1 > \lambda_2 + 1$.
    Then, consider the partition $\lambda' = ( \lambda_1 - 1, \lambda_2,
    \lambda_3, \ldots)$. By the assumption on $\lambda_1$, the partition
    $\lambda'$ still has distinct parts. On the other hand, the perimeter of
    $\lambda'$ is one less than the perimeter of $\lambda$. In fact,
    $\lambda'$ can be any of the $F_s$ many partitions into distinct parts
    with perimeter less than $s - 1$.
    
    \item The largest part $\lambda_1$ satisfies $\lambda_1 = \lambda_2 + 1$.
    In that case, consider the partition $\lambda' = ( \lambda_2, \lambda_3,
    \ldots)$, which has distinct parts. Clearly, the perimeter of $\lambda'$
    is two less than the perimeter of $\lambda$. Again, $\lambda'$ can be any
    of the $F_{s - 1}$ many partitions into distinct parts with perimeter less
    than $s - 2$.
  \end{enumerate}
  Taken together, we find that the number of partitions into distinct parts
  with perimeter less than $s$ is given by $F_s + F_{s - 1} = F_{s + 1}$, as
  claimed.
\end{proof}

Now, we establish Theorem~\ref{thm:distinct-odd} via a natural variation of
our proof of Theorem~\ref{thm:fib:distinct}. Recall that the perimeter of a
partition is the maximum hook length in the partition. In the spirit of
Euler's result, Theorem~\ref{thm:distinct-odd} claims that, for $M \geq
1$, the number of partitions into distinct parts with perimeter $M$ is equal
to the number of partitions into odd parts with perimeter $M$, and that this
common number is $F_M$.

\begin{proof}[Proof of Theorem~\ref{thm:distinct-odd}]
  By Lemma~\ref{lem:distinct:m1}, a
  partition into distinct parts is $( s, s + 1)$-core if and only if it has
  perimeter at most $s - 1$. Hence, Theorem~\ref{thm:fib:distinct} can be
  rephrased as saying that there are $F_{M + 2}$ many partitions into distinct
  parts with perimeter at most $M$. Consequently, there are $F_M = F_{M + 2} -
  F_{M + 1}$ many partitions into distinct parts with perimeter exactly $M$.
  This verifies the first part of Theorem~\ref{thm:distinct-odd}.
  
  It remains to prove that there are also $F_M$ partitions into odd parts with
  perimeter $M$. We proceed using a variation of our proof of
  Theorem~\ref{thm:fib:distinct}. Again, it is straightforward to verify the
  claim for $M = 1$ and $M = 2$. For the purpose of induction suppose that,
  for all $m < M$, there are $F_m$ many partitions into odd parts with
  perimeter $m$. Let $\lambda = ( \lambda_1, \lambda_2, \ldots)$ be a
  partition into odd parts with perimeter $M$. Then exactly one of the
  following two cases applies:
  \begin{enumerate}
    \item The largest part $\lambda_1$ satisfies $\lambda_1 > \lambda_2 + 1$.
    Then, consider the partition $\lambda' = ( \lambda_1 - 2, \lambda_2,
    \lambda_3, \ldots)$. Clearly, the parts of $\lambda'$ are all odd, and the
    perimeter of $\lambda'$ is $M - 2$. Evidently, $\lambda'$ can be any of
    the $F_{M - 2}$ many partitions into odd parts with perimeter $M - 2$.
    
    \item The largest part $\lambda_1$ satisfies $\lambda_1 = \lambda_2$. In
    that case, consider the partition $\lambda' = ( \lambda_2, \lambda_3,
    \ldots)$. Clearly, the parts of $\lambda'$ are all odd, and the perimeter
    of $\lambda'$ is $M - 1$. Again, $\lambda'$ can be any of the $F_{M - 1}$
    many partitions into odd parts with perimeter $M - 1$.
  \end{enumerate}
  Taken together, we find that the number of partitions into odd parts with
  perimeter $M$ is given by $F_{M - 2} + F_{M - 1} = F_M$.
\end{proof}

\section{An explicit bijection}\label{sec:euler}

Partition theorists are often interested in bijective proofs of statements of
equinumerosity. In the present discussion, a combination and comparison of the
recursive proofs of Theorems~\ref{thm:distinct-odd} and \ref{thm:fib:distinct}
does yield an explicit bijection between partitions into distinct parts with
perimeter $M$ and partitions into odd parts with perimeter $M$.

Let $\mathcal{C}$ be the set of all compositions with parts $1$ and $2$, and
such that the last part is not a $2$. For instance, $\mathcal{C}$ contains the
following compositions $\mu = ( \mu_1, \mu_2, \ldots)$ of $| \mu | = \mu_1 +
\mu_2 + \ldots = n$, for $n = 0, 1, \ldots 5$.

\begin{table}[h]
  \begin{equation*}
    \begin{array}{|l|l|l|}
       \hline
       \text{compositions $\mu$ in $\mathcal{C}$} & | \mu | & \#\\
       \hline
       () & 0 & 1\\
       \hline
       ( 1) & 1 & 1\\
       \hline
       ( 1, 1) & 2 & 1\\
       \hline
       ( 1, 1, 1), ( 2, 1) & 3 & 2\\
       \hline
       ( 1, 1, 1, 1), ( 1, 2, 1), ( 2, 1, 1) & 4 & 3\\
       \hline
     \end{array}
  \end{equation*}
  \caption{\label{tbl:C:small}Small compositions $\mu \in \mathcal{C}$ ranked
  by $| \mu |$}
\end{table}

It is straightforward to enumerate compositions of $M$ in $\mathcal{C}$.

\begin{lemma}
  \label{lem:C:F}For $M \geq 1$, there are $F_M$ many compositions $\mu
  \in \mathcal{C}$ with $| \mu | = M$.
\end{lemma}

For instance, a well-known equivalent version of this count appears as an
exercise in \cite[Chapter~1, Exercise 14(c)]{stanley-ec1}, where the reader
is asked to show that the number of compositions $\mu$ of $M$ into parts $1$
and $2$ is $F_{M + 1}$.

Next, we introduce bijections between $\mathcal{C}$ and the sets of partitions
into distinct (respectively, odd) parts. Combining these two bijections, we
then obtain a bijection, preserving perimeters, between partitions into
distinct parts and partitions into odd parts. In particular, this fact implies
Theorems~\ref{thm:distinct-odd} and \ref{thm:fib:distinct}.

\begin{theorem}
  The map $\mu \mapsto \lambda_d ( \mu)$, described below, is a bijection
  between $\mathcal{C}$ and the set of partitions into distinct parts.
  Likewise, the map $\mu \mapsto \lambda_o ( \mu)$ is a bijection between
  $\mathcal{C}$ and the set of partitions into odd parts. Moreover, the
  perimeter of the partitions $\lambda_d ( \mu)$ and $\lambda_o ( \mu)$ is $|
  \mu |$.
\end{theorem}

\begin{proof}
  Let $\mu \in \mathcal{C}$. We assign a partition $\lambda_d = \lambda_d (
  \mu)$ to $\mu$ by the following recursive recipe. If $\mu = ()$ or $\mu = (
  1)$, then $\lambda_d = \mu$. Otherwise, write $\mu = ( \mu_1, \mu')$ with
  $\mu_1 \in \{ 1, 2 \}$ and $\mu' \in \mathcal{C}$. Suppose that $\lambda' =
  ( \lambda_1, \lambda_2, \lambda_3, \ldots)$ is the partition assigned to
  $\mu'$.
  \begin{enumerate}
    \item If $\mu_1 = 1$, then $\lambda_d = ( \lambda_1 + 1, \lambda_2,
    \lambda_3, \ldots)$.
    
    \item If $\mu_1 = 2$, then $\lambda_d = ( \lambda_1 + 1, \lambda_1,
    \lambda_2, \lambda_3, \ldots)$.
  \end{enumerate}
  By construction, the partition $\lambda_d$ has distinct parts. In fact, it
  is straightforward to verify (in the spirit of the proof of
  Theorem~\ref{thm:fib:distinct}) that the map $\mu \mapsto \lambda_d ( \mu)$
  describes a bijection between $\mathcal{C}$ and the set of partitions into
  distinct parts.
  
  Analogously, we assign a partition $\lambda_o = \lambda_o ( \mu)$ to $\mu
  \in \mathcal{C}$ as follows. Again, if $\mu = ()$ or $\mu = ( 1)$, then
  $\lambda_o = \mu$. Otherwise, write $\mu = ( \mu_1, \mu')$ with $\mu_1 \in
  \{ 1, 2 \}$ and $\mu' \in \mathcal{C}$. Suppose that $\lambda' = (
  \lambda_1, \lambda_2, \lambda_3, \ldots)$ is the partition assigned to
  $\mu'$.
  \begin{enumerate}
    \item If $\mu_1 = 1$, then $\lambda_o = ( \lambda_1, \lambda_1, \lambda_2,
    \lambda_3, \ldots)$.
    
    \item If $\mu_1 = 2$, then $\lambda_o = ( \lambda_1 + 2, \lambda_2,
    \lambda_3, \ldots)$.
  \end{enumerate}
  Then, the map $\mu \mapsto \lambda_o ( \mu)$ describes a bijection between
  $\mathcal{C}$ and the set of partitions into odd parts.
  
  Combining these two bijections, we have a bijection between partitions
  $\lambda_d$ into distinct parts and partitions $\lambda_o$ into odd parts.
  \begin{equation*}
    \lambda_d = \lambda_d ( \mu) \quad \leftrightarrow \quad
     \mu \quad \leftrightarrow \quad \lambda_o ( \mu) =
     \lambda_o
  \end{equation*}
  It also follows from the respective constructions that the partitions
  $\lambda_d ( \mu)$ and $\lambda_o ( \mu)$ both have perimeter $| \mu |$.
\end{proof}

\begin{example}
  The following table lists all $13$ partitions into distinct (respectively,
  odd) parts with perimeter at most $5$, together with the composition $\mu$
  in $\mathcal{C}$ they get matched with.
  \begin{equation*}
    \begin{array}{|l|l|l|}
       \hline
       \mu & \lambda_d & \lambda_o\\
       \hline
       () & () & ()\\
       \hline
       ( 1) & ( 1) & ( 1)\\
       \hline
       ( 1, 1) & ( 2) & ( 1, 1)\\
       \hline
       ( 1, 1, 1) & ( 3) & ( 1, 1, 1)\\
       ( 2, 1) & ( 2, 1) & ( 3)\\
       \hline
       ( 1, 1, 1, 1) & ( 4) & ( 1, 1, 1, 1)\\
       ( 1, 2, 1) & ( 3, 1) & ( 3, 3)\\
       ( 2, 1, 1) & ( 3, 2) & ( 3, 1)\\
       \hline
       ( 1, 1, 1, 1, 1) & ( 5) & ( 1, 1, 1, 1, 1)\\
       ( 1, 1, 2, 1) & ( 4, 1) & ( 3, 3, 3)\\
       ( 1, 2, 1, 1) & ( 4, 2) & ( 3, 3, 1)\\
       ( 2, 1, 1, 1) & ( 4, 3) & ( 3, 1, 1)\\
       ( 2, 2, 1) & ( 3, 2, 1) & ( 5)\\
       \hline
     \end{array}
  \end{equation*}
  In particular, note that the composition $\mu = ( 1, 2, 1)$ corresponds to
  the partitions
  \begin{eqnarray*}
    \lambda_d ( \mu) & = & ( 3, 1),\\
    \lambda_o ( \mu) & = & ( 3, 3) .
  \end{eqnarray*}
  This illustrates that, while the present bijection between partitions into
  distinct parts and partitions into odd parts preserves the perimeter, it
  does not preserve the size of the partitions. Therefore, our bijection is of
  a rather different nature compared to the bijections underlying Euler's
  theorem and its generalizations.
\end{example}

\section{A generalization}\label{sec:fibx}

Recall that we showed in Theorem~\ref{thm:fib:distinct} that, as conjectured
by Amdeberhan \cite{amdeberhan-conj}, $( s - 1, s)$-core partitions into
distinct parts are counted by the Fibonacci numbers $F_s$. In this section, we
generalize this result and enumerate $( s, d s - 1)$-core partitions into
distinct parts for any $d \geq 1$. That is, we prove
Theorem~\ref{thm:fibx:distinct} from the introduction, which is restated here
for the reader's convenience.

\begin{theorem}
  \label{thm:fibx:distinct2}The number $N_d ( s)$ of $( s, d s - 1)$-core
  partitions into distinct parts is characterized by $N_d ( 1) = 1$, $N_d ( 2)
  = d$ and, for $s \geq 3$,
  \begin{equation}
    N_d ( s) = N_d ( s - 1) + d N_d ( s - 2) . \label{eq:core:rec}
  \end{equation}
\end{theorem}

We are confident that a suitable generalization of our proof of
Theorem~\ref{thm:fib:distinct} can be used to prove this result. \ Since the
details appear to be somewhat more technical, we instead offer an alternative
proof, inspired by the approach taken in \cite{xiong-cores}.

\begin{example}
  Versions of the numbers $N_d ( s)$ in Theorem~\ref{thm:fibx:distinct2} have
  been studied in the literature since Lucas, and are usually referred to as
  generalized Fibonacci numbers or generalized Fibonacci polynomials (in the
  variable $d$). For further information and references, we refer the
  interested reader to the recent paper \cite{acms-fibo}. The first few
  polynomials $N_d ( s)$, for $s = 1, 2, \ldots, 7$, are
  \begin{equation*}
    1, \quad d, \quad 2 d, \quad d ( d + 2) 
     , \quad d ( 3 d + 2), \quad d ( d^2 + 5 d + 2),
     \quad d ( 4 d^2 + 7 d + 2) .
  \end{equation*}
  Of course, we recover the usual Fibonacci numbers upon setting $d = 1$.
\end{example}

In preparation for the proof of Theorem~\ref{thm:fibx:distinct2}, we say that
a set $X \subseteq \mathbb{Z}$ is {\emph{twin-free}} if there is no $x \in X$
such that $\{ x, x + 1 \} \subseteq X$. As the following result shows, the
number of tuples of nested twin-free sets satisfies the same recursive
relation that is claimed for the core partitions in
Theorem~\ref{thm:fibx:distinct2}. Note, however, that the initial conditions
differ.

\begin{lemma}
  \label{lem:twinfree}Let $M_d ( s)$ denote the number of tuples $( X_1, X_2,
  \ldots, X_d)$ of twin-free sets such that $X_d \subseteq X_{d - 1} \subseteq
  \ldots \subseteq X_1 \subseteq \{ 1, 2, \ldots, s - 1 \}$. Then, $M_d ( 1) =
  1$, $M_d ( 2) = d + 1$ and, for $s \geq 3$,
  \begin{equation}
    M_d ( s) = M_d ( s - 1) + d M_d ( s - 2) . \label{eq:twinfree:rec}
  \end{equation}
\end{lemma}

\begin{proof}
  Clearly, $M_d ( 1) = 1$ because in that case the only tuple $( X_1, X_2,
  \ldots, X_d)$ is the one with $X_1 = X_2 = \ldots = X_d = \{ \}$. On the
  other hand, $M_d ( 2) = d + 1$ because then all tuples $( X_1, X_2, \ldots,
  X_d)$ are of the form $X_j = \{ 1 \}$, if $j \leq J$, and $X_j = \{
  \}$, if $j > J$, for some $J \in \{ 0, 1, \ldots, d \}$.
  
  We may therefore suppose that $s \geq 2$. Let $( X_1, X_2, \ldots,
  X_d)$ be a tuple of twin-free sets such that $X_d \subseteq X_{d - 1}
  \subseteq \ldots \subseteq X_1 \subseteq \{ 1, 2, \ldots, s - 1 \}$. Then,
  exactly one of the following two possibilities is true:
  \begin{enumerate}
    \item \label{i:twinfree:1}None of the sets $X_1, X_2, \ldots, X_d$
    contains $s - 1$.
    
    \item \label{i:twinfree:2}There is an index $J \in \{ 1, 2, \ldots, d \}$
    such that $s - 1 \in X_j$ for all $j \leq J$ and $s - 1 \nin X_j$ for
    all $j > J$.
  \end{enumerate}
  In case \ref{i:twinfree:1}, our tuple $( X_1, X_2, \ldots, X_d)$ is one of
  the $M_d ( s - 1)$ many tuples of twin-free sets such that $X_d \subseteq
  X_{d - 1} \subseteq \ldots \subseteq X_1 \subseteq \{ 1, 2, \ldots, s - 2
  \}$.
  
  On the other hand, suppose case \ref{i:twinfree:2} holds with $J \in \{ 1,
  2, \ldots, d \}$. In that case, $s - 1 \in X_1$. Since $X_1$ is twin-free it
  follows that $s - 2 \nin X_1$, and hence $s - 2$ is not contained in any of
  the sets $X_1, X_2, \ldots, X_d$. Let $X_1', X_2', \ldots, X_d'$ be the sets
  obtained from $X_1, X_2, \ldots, X_d$ by removing $s - 1$ from these sets.
  That is, $X_j' = X_j - \{ s - 1 \}$. Observe that the tuple $( X_1', X_2',
  \ldots, X_d')$ can be any of the $M_d ( s - 2)$ many tuples of twin-free
  sets such that $X_d' \subseteq X_{d - 1}' \subseteq \ldots \subseteq X_1'
  \subseteq \{ 1, 2, \ldots, s - 3 \}$. Since $( X_1', X_2', \ldots, X_d')$
  together with the value of $J$ determines $( X_1, X_2, \ldots, X_d)$, we
  conclude that case \ref{i:twinfree:2} accounts for exactly $d M_d ( s - 2)$
  many tuples.
  
  The recursive relation \eqref{eq:twinfree:rec} follows upon combining these
  two cases.
\end{proof}

\begin{lemma}
  \label{lem:twinfree:core}$( s, d s - 1)$-core partitions into distinct parts
  are in bijective correspondence with tuples $( X_1, X_2, \ldots, X_d)$ of
  twin-free sets such that $X_d \subseteq X_{d - 1} \subseteq \ldots \subseteq
  X_1 \subseteq \{ 1, 2, \ldots, s - 1 \}$ and $s - 1 \nin X_d$.
\end{lemma}

\begin{proof}
  Following \cite{xiong-cores}, given a partition $\lambda$, we denote with
  $\beta ( \lambda)$ the set of hook lengths $h ( u)$ where $u$ is a cell in
  the first column of $\lambda$. Clearly, the set $\beta ( \lambda)$ uniquely
  determines $\lambda$. Moreover, $\lambda$ is $t$-core if and only if, for
  any $x \in \beta ( \lambda)$ with $x \geq t$, we always have $x - t \in
  \beta ( \lambda)$ \cite[Lemma~2.1]{xiong-cores}. In particular, if
  $\lambda$ is $t$-core, then $\lambda$ is also $n t$-core for any $n
  \geq 1$. This implies that an $( s, d s - 1)$-core partition into
  distinct parts is also $( d s - 1, d s)$-core and hence, by
  Lemma~\ref{lem:distinct:m1}, has perimeter (maximum hook length) at most $d
  s - 2$.
  
  Let $\lambda$ be an $( s, d s - 1)$-core partition into distinct parts.
  Equivalently, $\lambda$ is an $s$-core partition into distinct parts with
  perimeter at most $d s - 2$. Therefore, the set $\beta ( \lambda)$ can be
  any twin-free set
  \begin{equation*}
    \beta ( \lambda) = \beta_1 ( \lambda) \cup \beta_2 ( \lambda) \cup \ldots
     \cup \beta_{d - 1} ( \lambda) \cup \beta_d ( \lambda)
  \end{equation*}
  where
  \begin{equation*}
    \beta_j ( \lambda) \subseteq \{ ( j - 1) s + 1, ( j - 1) s + 2, \ldots, j
     s - 1 \},
  \end{equation*}
  for $j \in \{ 1, 2, \ldots, d - 1 \}$, and
  \begin{equation*}
    \beta_d ( \lambda) \subseteq \{ ( d - 1) s + 1, ( d - 1) s + 2, \ldots, d
     s - 2 \} .
  \end{equation*}
  Attach to $\lambda$ the tuple $( X_1, X_2, \ldots, X_d)$ with
  \begin{equation*}
    X_j = \left\{ x - ( j - 1) s \; : \; x \in \beta_j (
     \lambda) \right\} .
  \end{equation*}
  By construction, $X_j \subseteq \{ 1, 2, \ldots, s - 1 \}$. Since $\lambda$
  has distinct parts, the sets $X_j$ are all twin-free. Recall that the
  condition that $\lambda$ is $s$-core is equivalent to the following: if $x
  \in \beta_j ( \lambda)$ with $j > 1$, then $x - s \in \beta_{j - 1} (
  \lambda)$. This translates into $X_d \subseteq X_{d - 1} \subseteq \ldots
  \subseteq X_1$. Finally, $s - 1 \nin X_d$ because $d s - 1 \nin \beta (
  \lambda)$. Since there are no further restrictions on the sets $X_j$, we
  have arrived at the bijective correspondence, as promised.
\end{proof}

Now, we are in a comfortable position to prove
Theorem~\ref{thm:fibx:distinct2}.

\begin{proof}[Proof of Theorem~\ref{thm:fibx:distinct2}]
  In light of the bijective
  correspondence established in Lemma~\ref{lem:twinfree:core}, $N_d ( s)$
  equals the number of tuples $( X_1, X_2, \ldots, X_d)$ of twin-free sets
  such that $X_d \subseteq X_{d - 1} \subseteq \ldots \subseteq X_1 \subseteq
  \{ 1, 2, \ldots, s - 1 \}$ and $s - 1 \nin X_d$. We need to show that $N_d (
  1) = 1$, $N_d ( 2) = d$ and $N_d ( s) = N_d ( s - 1) + d N_d ( s - 2)$. This
  is clearly a variation of Lemma~\ref{lem:twinfree} and, indeed, we can prove
  it along the same lines.
  
  As in the proof of Lemma~\ref{lem:twinfree}, we see that $N_d ( 1) = 1$,
  $N_d ( 2) = d$, the only difference being that in the latter case one tuple
  is excluded due to the condition $s - 1 \nin X_d$. Therefore, consider the
  case $s \geq 2$. Let $( X_1, X_2, \ldots, X_d)$ be a tuple of twin-free
  sets such that $X_d \subseteq X_{d - 1} \subseteq \ldots \subseteq X_1
  \subseteq \{ 1, 2, \ldots, s - 1 \}$ and $s - 1 \nin X_d$. Then, exactly one
  of the following two possibilities is true:
  \begin{enumerate}
    \item \label{i:twinfree:core:1}None of the sets $X_1, X_2, \ldots, X_d$
    contains $s - 1$.
    
    \item \label{i:twinfree:core:2}There is an index $J \in \{ 1, 2, \ldots, d
    - 1 \}$ such that $s - 1 \in X_j$ for all $j \leq J$ and $s - 1 \nin
    X_j$ for all $j > J$.
  \end{enumerate}
  In case \ref{i:twinfree:core:1}, our tuple $( X_1, X_2, \ldots, X_d)$ is one
  of the $M_d ( s - 1)$ many tuples from Lemma~\ref{lem:twinfree} of twin-free
  sets such that $X_d \subseteq X_{d - 1} \subseteq \ldots \subseteq X_1
  \subseteq \{ 1, 2, \ldots, s - 2 \}$.
  
  On the other hand, suppose case \ref{i:twinfree:core:2} holds with $J \in \{
  1, 2, \ldots, d - 1 \}$. As in the proof of Lemma~\ref{lem:twinfree}, let
  $X_j' = X_j - \{ s - 1 \}$. Again, the resulting tuple $( X_1', X_2',
  \ldots, X_d')$ can be any of the $M_d ( s - 2)$ many tuples of twin-free
  sets such that $X_d' \subseteq X_{d - 1}' \subseteq \ldots \subseteq X_1'
  \subseteq \{ 1, 2, \ldots, s - 3 \}$. Since $( X_1', X_2', \ldots, X_d')$
  together with the value of $J$ determines $( X_1, X_2, \ldots, X_d)$, we
  conclude that case \ref{i:twinfree:core:2} accounts for exactly $( d - 1)
  M_d ( s - 2)$ many tuples.
  
  Combining these two cases, we arrive at
  \begin{equation}
    N_d ( s) = M_d ( s - 1) + ( d - 1) M_d ( s - 2) . \label{eq:core:twinfree}
  \end{equation}
  The asserted recurrence relation \eqref{eq:core:rec} for $N_d ( s)$
  therefore follows from the recurrence relation \eqref{eq:twinfree:rec} for
  $M_d ( s)$. Indeed, for all $s \geq 5$,
  \begin{eqnarray*}
    N_d ( s) & = & [ M_d ( s - 2) + d M_d ( s - 3)] + ( d - 1) [ M_d ( s - 3)
    + d M_d ( s - 4)]\\
    & = & [ M_d ( s - 2) + ( d - 1) M_d ( s - 3)] + d [ M_d ( s - 3) + ( d -
    1) M_d ( s - 4)]\\
    & = & N_d ( s - 1) + d N_d ( s - 2) .
  \end{eqnarray*}
  It only remains to verify initial values. By \eqref{eq:core:twinfree}, we
  have $N_d ( 3) = M_d ( 2) + ( d - 1) M_d ( 1) = 2 d$ and $N_d ( 4) = M_d (
  3) + ( d - 1) M_d ( 2) = d M_d ( 2) + d M_d ( 1) = d^2 + 2 d$. These values
  indeed also satisfy the recursive relation \eqref{eq:core:rec} for $N_d (
  s)$, since $N_d ( 2) + d N_d ( 1) = 2 d$ and $N_d ( 3) + d N_d ( 2) = d^2 +
  2 d$.
\end{proof}

\section{Conclusion}

We proved an analog for partitions into distinct parts of Anderson's
Theorem~\ref{thm:anderson} specialized to $( s, s + 1)$-core partitions. More
generally, in Theorem~\ref{thm:fibx:distinct}, we enumerated $( s, d s -
1)$-core partitions into distinct parts. It would be interesting to further
generalize this result and determine a count for $( s, t)$-core partitions
into distinct parts, for any coprime $s$ and $t$. In this direction, we offer
the following conjecture for further motivation.

\begin{conjecture}
  \label{conj:2}If $s$ is odd, then the number of $( s, s + 2)$-core
  partitions into distinct parts equals $2^{s - 1}$.
\end{conjecture}

This claim is based on experimental evidence and has been verified for $s <
20$ after listing all relevant partitions.

\begin{example}
  For $s = 3$, the four $( 3, 5)$-core partitions into distinct parts are
  \begin{equation*}
    \{ \} , \quad \{ 1 \}, \quad \{ 2 \}, \quad
     \{ 3, 1 \} .
  \end{equation*}
  For $s = 5$, the sixteen $( 5, 7)$-core partitions into distinct parts are
  \begin{align*}
    & \{ \} , \quad \{ 1 \}, \quad \{ 2 \},
    \quad \{ 3 \}, \quad \{ 4 \} , \quad \{ 2, 1
    \}, \quad \{ 3, 1 \}, \quad \{ 5, 1 \}, \quad \{ 3, 2
    \}, \quad \{ 4, 2, 1 \},\\
    & \{ 6, 2, 1 \}, \quad \{ 4, 3, 1 \}, \quad \{ 7, 3, 2
    \}, \quad \{ 5, 4, 2, 1 \}, \quad \{ 8, 4, 3, 1 \},
    \quad \{ 9, 5, 4, 2, 1 \} .
  \end{align*}
  Note that the largest occurring size among these partitions is $3 + 1 = 4$,
  for $s = 3$, and $9 + 5 + 4 + 2 + 1 = 21$, for $s = 5$. For $s = 3, 5,
  \ldots, 17$, the largest possible sizes of $( s, s + 2)$-core partitions
  into distinct parts are
  \begin{equation*}
    4, 21, 65, 155, 315, 574, 966, 1530.
  \end{equation*}
  Based on the initial data, it appears that there is a unique partition of
  this largest size, and that the largest possible size of an $( s, s +
  2)$-core partition into distinct parts is $\frac{1}{384} ( s^2 - 1) ( s + 3)
  ( 5 s + 17)$. This partition of largest size appears to have both the
  highest number of parts (namely, $\frac{1}{8} ( s - 1) ( s + 5)$ many) and
  the largest part (namely, a part of size $\frac{3}{8} ( s^2 - 1)$). After
  $\{ 3, 1 \}$ and $\{ 9, 5, 4, 2, 1 \}$, the next such unique largest
  partitions are
  \begin{equation*}
    \{18, 12, 11, 7, 6, 5, 3, 2, 1\} , \quad \{30, 22, 21, 15,
     14, 13, 9, 8, 7, 6, 4, 3, 2, 1\}.
  \end{equation*}
\end{example}

We hope that Conjecture~\ref{conj:2} together with the results in this paper
provide clues for enumerating $( s, t)$-core partitions into distinct parts.
Table~\ref{tbl:stcore:distinct} lists the number of such partitions for $s, t
\leq 12$. Observe, in particular, the occurrence of the Fibonacci numbers
next to the main diagonal, in accordance with Theorem~\ref{thm:fib:distinct}.

\begin{table}[h]
  \begin{equation*}
    \begin{array}{|c||c|c|c|c|c|c|c|c|c|c|c|c|}
       \hline
       s\backslash t & 1 & 2 & 3 & 4 & 5 & 6 & 7 & 8 & 9 & 10 & 11 & 12\\
       \hhline{|=#=|=|=|=|=|=|=|=|=|=|=|=|}
       1 & 1 & 1 & 1 & 1 & 1 & 1 & 1 & 1 & 1 & 1 & 1 & 1\\
       \hline
       2 & 1 & \infty & 2 & \infty & 3 & \infty & 4 & \infty & 5 & \infty & 6
       & \infty\\
       \hline
       3 & 1 & 2 & \infty & 3 & 4 & \infty & 5 & 6 & \infty & 7 & 8 & \infty\\
       \hline
       4 & 1 & \infty & 3 & \infty & 5 & \infty & 8 & \infty & 11 & \infty &
       15 & \infty\\
       \hline
       5 & 1 & 3 & 4 & 5 & \infty & 8 & 16 & 18 & 16 & \infty & 21 & 38\\
       \hline
       6 & 1 & \infty & \infty & \infty & 8 & \infty & 13 & \infty & \infty &
       \infty & 32 & \infty\\
       \hline
       7 & 1 & 4 & 5 & 8 & 16 & 13 & \infty & 21 & 64 & 50 & 64 & 114\\
       \hline
       8 & 1 & \infty & 6 & \infty & 18 & \infty & 21 & \infty & 34 & \infty &
       101 & \infty\\
       \hline
       9 & 1 & 5 & \infty & 11 & 16 & \infty & 64 & 34 & \infty & 55 & 256 &
       \infty\\
       \hline
       10 & 1 & \infty & 7 & \infty & \infty & \infty & 50 & \infty & 55 &
       \infty & 89 & \infty\\
       \hline
       11 & 1 & 6 & 8 & 15 & 21 & 32 & 64 & 101 & 256 & 89 & \infty & 144\\
       \hline
       12 & 1 & \infty & \infty & \infty & 38 & \infty & 114 & \infty & \infty
       & \infty & 144 & \infty\\
       \hline
     \end{array}
  \end{equation*}
  \caption{\label{tbl:stcore:distinct}The number of $( s, t)$-core partitions
  into distinct parts for $s, t \leq 12$}
\end{table}

It would further be interesting, but appears to be harder, to enumerate $( s,
t)$-core partitions into odd parts.

\begin{acknowledgements}
I thank Tewodros Amdeberhan for introducing me to
$( s, t)$-core partitions and his conjecture, as well as for many interesting
discussions, comments and suggestions. I am also grateful to
George~E.~Andrews, Bruce~C.~Berndt, Robert Osburn and Wadim Zudilin for
comments on an earlier version of this paper.
\end{acknowledgements}

\end{document}